%% file: main.tex
\documentclass[11pt,dvipsnames]{amsart}
\usepackage{amsmath,amssymb,amsfonts}
\usepackage[mathscr]{eucal}
\usepackage[margin=1in]{geometry}
\usepackage{quiver}
\usepackage{comment}
\usepackage{mathtools}

\usepackage{tikz, tikz-cd}%
\usetikzlibrary{matrix,arrows,decorations.pathmorphing,cd,patterns,calc,backgrounds,patterns}
\tikzset{dummy/.style= {circle,fill,draw,inner sep=0pt,minimum size=1.2mm}}
\tikzset{vertex/.style={fill, circle, minimum size=.1cm, inner sep=0pt}}

\usepackage[
colorlinks, citecolor=MidnightBlue, linkcolor=black, urlcolor=MidnightBlue]{hyperref}
\hypersetup{linktoc=all,} 

\usepackage[textwidth=25mm, textsize=tiny]{todonotes}
\usepackage[final]{microtype}

\usepackage{enumerate}

\usepackage{appendix}

\numberwithin{equation}{section} 
\numberwithin{figure}{section}
\usepackage[nameinlink,capitalise,noabbrev]{cleveref}

\newcommand{\newrefformat}[2]{}

\usepackage{stmaryrd}

\newcommand\restr[2]{{
  \left.\kern-\nulldelimiterspace 
  #1 
  \vphantom{\big|} 
  \right|_{#2} 
  }}

\usepackage[
backend=biber,
style=alphabetic,
maxbibnames=99,
url = false,
doi = false,
isbn = false,
]{biblatex}
\crefname{lemma}{Lemma}{Lemmas}
\crefname{theorem}{Theorem}{Theorems}
\crefname{definition}{Definition}{Definitions}
\crefname{proposition}{Proposition}{Propositions}
\crefname{remark}{Remark}{Remarks}
\crefname{corollary}{Corollary}{Corollaries}
\crefname{equation}{Equation}{Equations}
\crefname{construction}{Construction}{Constructions}
\crefname{ex}{Example}{Examples}
\crefname{appsec}{Appendix}{Appendices}
\crefname{subsection}{Subsection}{Subsections}

\theoremstyle{plain}
\newtheorem{theorem}[equation]{Theorem}
\newtheorem{corollary}[equation]{Corollary}
\newtheorem{proposition}[equation]{Proposition}
\newtheorem{lemma}[equation]{Lemma}

\newtheorem*{theorem*}{Theorem}
\newtheorem{introtheorem}{Theorem}
 
\crefname{introtheorem}{Theorem}{Theorems}

\theoremstyle{definition}
\newtheorem{definition}[equation]{Definition}
\newtheorem{example}[equation]{Example}
\newtheorem{remark}[equation]{Remark}

\author[M. E. Calle]{Maxine E. Calle}             
\address{Department of Mathematics,
          University of Pennsylvania,
          Philadelphia, PA, 19104,
          USA}
\email{callem@sas.upenn.edu}

\author[D. Chan]{David Chan}
\address{Department of Mathematics,
         Michigan State University,
         East Lansing, MI, 48824
         USA}
\email{chandav2@msu.edu}



\usepackage[textwidth=25mm, textsize=tiny]{todonotes}

\title[A genuine $G$-spectrum for the cut-and-paste $K$-theory of $G$-manifolds]{A genuine $G$-spectrum for the\\cut-and-paste $K$-theory of $G$-manifolds}
\date{\today}

\begin{document}
\include{shortcuts}

\begin{abstract}
Recent work has applied scissors congruence $K$-theory to study classical cut-and-paste ($SK$) invariants of manifolds.
This paper proves the conjecture that the squares $K$-theory of equivariant $SK$-manifolds arises as the fixed points of a genuine $G$-spectrum. Our method utilizes the framework of spectral Mackey functors as models for genuine $G$-spectra, and our main technical result is a general procedure for constructing spectral Mackey functors using squares $K$-theory.
\end{abstract}
\maketitle
\vspace{-1cm}

\section{Introduction}
Given two $d$-dimensional manifolds $M$ and $N$, one can ask whether it is possible to cut $M$ up into pieces and reassemble these pieces in a new way to obtain $N$. This ``cut-and-paste'' relation, also known as an \textit{$SK$-relation} (which abbreviates the German translation \textit{schneiden und kleben} for cut-and-paste), is surprisingly computable.
In \cite{KKNO73}, Karras--Kreck--Neumann--Ossa show that the $SK$-equivalence classes of a given dimension are completely classified by algebraic invariants, namely, the Euler characteristic and (in $4d$-dimensions) the signature.

In the setting of $G$-manifolds, one may similarly make sense of an equivariant $SK$-relation and study the resulting $SK$-invariants, see for example \cite{rowlett,komiya,kosniowski,hara-koshikawa}. However, in this context, the story is not as simple nor as well-understood. For instance, the equivariant Euler characteristic, while still an $SK$-invariant, is no longer a complete invariant (see \cite[Example 2.14]{SKG}). While equivariant $SK$-invariants have been completely classified for some groups (see e.g. \cite{kosniowski}) it is still an open problem in general.

The study of $SK$-invariants is similar in flavor to \textit{scissors congruence} invariants of polytopes.  Both theories fit into a framework called \textit{scissors congruence $K$-theory}, pioneered by Zakharevich \cite{Zak-perspectives}, which takes inspiration from the rich subject of the algebraic $K$-theory of rings. The connection between $SK$-invariants and scissors congruence $K$-theory was first realized in \cite{HMMRS}, where the authors construct a \textit{category with squares} ${\rm Mfld}^{\bndry}_d$ whose $K$-theory recovers $SK$-groups for manifolds with boundary as its zeroth $K$-group (see \cref{ex:SK}). The scissors congruence $K$-theory of $SK$-manifolds is obtained via a construction $K^{\square}$ called \textit{squares $K$-theory} \cite{CKMZ:squares}. 

In a recent paper \cite{SKG}, these ideas are extended to the equivariant context. The authors construct a category with squares ${\rm Mfld}^{G,\bndry}_d$ which is the equivariant analog of the one from \cite{HMMRS} and they show that $K_0^{\square}({\rm Mfld}^{G,\bndry}_d)\cong \SK_d^G$ is an $SK$-group for $G$-manifolds with boundary.
Moreover, they show that the groups $\{\SK_d^H\}_{H\leq G}$ assemble into a \emph{$G$-Mackey functor}, an equivariant algebraic structure which encodes the relationships between the groups $\SK^H_d$ as $H\leq G$ varies. Since $G$-Mackey functors naturally arise as the homotopy groups of genuine $G$-spectra, the authors of \cite{SKG} conjecture that their Mackey functor is in fact the $\pi_0$-Mackey functor of some genuine $G$-spectrum. Our main result is the affirmation of this conjecture.

\begin{introtheorem}[\cref{thm:main}]\label{introthm:main}
        There is a genuine $G$-spectrum $K^\square_G(\cat M^G_d)$ such that \[
    K^{\square}_G(\cat M^G_d)^H \simeq K^{\square}({\rm Mfld}^{H,\bndry}_d)
    \] for all $H\leq G$. Moreover, $\pi_0$ of this genuine $G$-spectrum is the Mackey functor of equivariant $SK$-groups for $d$-dimensional $G$-manifolds with boundary. 
\end{introtheorem}

Among other things, the theorem implies that the higher cut-and-paste $K$-theory groups of $G$-manifolds also assemble into Mackey functors, generalizing the results of \cite{SKG}. These higher cut-and-paste $K$-theory groups are related to cobordism and other invariants of manifolds \cite{MRS}.

The proof of \cref{introthm:main} relies on the framework of ($\infty$-categorical) spectral Mackey functors as a model for genuine $G$-spectra \cite{barwick_spectral_2017, barwick_spectral_2020} (see also \cite{GuillouMay}). The essential idea is to construct a Mackey functor in categories with squares and then apply $K^{\square}$ levelwise to obtain a spectral Mackey functor, in the spirit of \cite{BohmannOsorno1,barwick_spectral_2017, barwick_spectral_2020, MalkiewichMerling:Atheory,linearization}. 
Our main technical result is the following, which is explained in more detail in \cref{sec:constructing SpMacks}. We note that we take $\infty$-category to mean quasicategory.

\begin{introtheorem}[\cref{thm:KT machine,corollary: from mackey functor of squares cats to spectral Mackey}]\label{introthm:sq G-KT}
    Squares $K$-theory extends to a product-preserving $\infty$-functor from an $\infty$-category of categories with squares to spectra.  In particular, every Mackey functor in the $\infty$-category of squares categories determines a spectral Mackey functor.
\end{introtheorem}

With this $K^{\square}$-theory machine in hand, the bulk of the work that goes into \cref{introthm:main} is to construct a Mackey functor of squares categories. We note that while the statement of \cref{introthm:sq G-KT} uses the language of $\infty$-categories, our construction of the main example is implemented using $2$-categories.  This is done in \cref{sec:SKG example}, although we emphasize that many of the main ideas were already introduced in \cite[Section 3]{SKG}. 

The second part of the conjecture from \cite{SKG} concerns the equivariant Euler characteristic, which is the primary example of an equivariant $SK$-invariant of $G$-manifolds. In contrast to the non-equivariant case, it is still an open question which $SK$-invariants completely classify equivariant cut-and-paste manifolds, although the answer is known when $G$ is a finite Abelian group of odd order \cite{kosniowski}. In \cite[Theorem B]{SKG}, the authors show that there is a map of $K$-theory spectra which recovers the equivariant Euler characteristic on $\pi_0$, and they conjecture that this map should arise as the fixed points of some map of genuine $G$-spectra. In \cref{sec:Euler char}, we show that this is indeed the case. In the theorem below, $K_G(\underline{\ZZ})$ denotes the equivariant algebraic $K$-theory of the constant coefficient system on $\ZZ$, as defined in \cite{linearization}. 

\begin{introtheorem}[\cref{theorem: lift euler char}]\label{introthm:linearization}
    There is a map of genuine $G$-spectra\[
    K_G^{\square}(\cat M_d^G) \to K_G(\underline{\ZZ})
    \] which on $\pi_0^H$ recovers the $H$-equivariant Euler characteristic, for $H\leq G$. 
\end{introtheorem}

The paper is organized as follows. In \cref{sec:sq KT}, we recall the basics of squares $K$-theory and categories with squares. In \cref{sec:SpMacks}, we review the necessary definitions and results about Mackey functors that we need to use to construct our equivariant $K^{\square}$-theory functor (which is done in \cref{sec:constructing SpMacks}). \cref{sec:SKG example} contains the proof of \cref{introthm:main} and \cref{sec:Euler char} contains that of \cref{introthm:linearization}.

\subsection*{Acknowledgments}
We are grateful to the authors of \cite{SKG} for fruitful conversations that led to the contents of this paper. We thank Liam Keenan, Teena Gerhardt, and Maximilien P\'eroux for helpful comments on an early draft, as well as the anonymous referee for their comments which greatly improved the paper.
The first-named author was partially supported by NSF grant DGE-1845298 and the second-named author was partially supported by NSF grant DMS-2135960. 
\section{Squares $K$-theory}\label{sec:sq KT}

In this section, we review the basic definitions of squares $K$-theory. We note that the categories we consider are the categories with squares from \cite{HMMRS}, which are not as general as the squares categories of \cite{CKMZ:squares}. 

\begin{definition}\label{defn:cat w squares}
A \textit{category with squares} consists of a category $\cat C$ with coproducts, a chosen distinguished object $O$, and two subcategories $\cat C^h$ and $\cat C^v$ of horizontal ($\rightarrowtail$) and vertical ($\twoheadrightarrow$) morphisms along with a collection of distinguished squares
\[
    \begin{tikzcd}
    A \ar[r, >->] \ar[d, ->>] \ar[dr, phantom, "\square"] & B \ar[d, ->>]\\
    C \ar[r, >->]& D
    \end{tikzcd}
\] 
which satisfy the following conditions:\begin{enumerate}
\item[(i)] distinguished squares are closed under coproducts,
    \[
        \text{if }~~~ \begin{tikzcd}
            A \ar[r, >->] \ar[d, ->>]  \ar[dr, phantom, "\square"]& B \ar[d, ->>]\\
            C \ar[r, >->]& D
            \end{tikzcd} ~~~\text{ and }~~~   \begin{tikzcd}
            A' \ar[dr, phantom, "\square"] \ar[r, >->] \ar[d, ->>] & B' \ar[d, ->>]\\
            C' \ar[r, >->]& D'
            \end{tikzcd}, ~~~\text{ then }~~~ \begin{tikzcd}
            A\coprod A' \ar[dr, phantom, "\square"] \ar[r, >->] \ar[d, ->>] & B\coprod B' \ar[d, ->>]\\
            C\coprod C' \ar[r, >->]& D\coprod D'
        \end{tikzcd};
    \]
\item[(ii)] distinguished squares are commutative in $\cat C$ and can be composed vertically and horizontally,
    \[
        \text{horizontally: if }~~~ \begin{tikzcd}
            A \ar[r, >->] \ar[d, ->>]  \ar[dr, phantom, "\square"]& B \ar[d, ->>]\\
            C \ar[r, >->]& D
            \end{tikzcd} ~~~\text{ and }~~~   \begin{tikzcd}
            B \ar[dr, phantom, "\square"] \ar[r, >->] \ar[d, ->>] & B' \ar[d, ->>]\\
            D \ar[r, >->]& D'
            \end{tikzcd}, ~~~\text{ then }~~~ \begin{tikzcd}
            A \ar[r, >->] \ar[d, ->>]\ar[drr, phantom, "\square"] & B \ar[r, >->] & B' \ar[d, ->>]\\
            C \ar[r, >->]& D\ar[r, >->] & D'
        \end{tikzcd};
    \]
    \[
        \text{vertically: if }~~~ \begin{tikzcd}
            A \ar[r, >->] \ar[d, ->>]  \ar[dr, phantom, "\square"]& B \ar[d, ->>]\\
            C \ar[r, >->]& D
            \end{tikzcd} ~~~\text{ and }~~~   \begin{tikzcd}
            C \ar[dr, phantom, "\square"] \ar[r, >->] \ar[d, ->>] & D \ar[d, ->>]\\
            C' \ar[r, >->]& D'
            \end{tikzcd}, ~~~\text{ then }~~~ \begin{tikzcd}
            A \ar[r, >->] \ar[d, ->>]\ar[ddr, phantom, "\square"] & B \ar[d, ->>] \\
            C \ar[d, ->>] & D\ar[d, ->>] \\
            C' \ar[r, >->] & D'
        \end{tikzcd};
    \]
\item[(iii)] the subcategory $iso\cat C$ of isomorphisms ($\xrightarrow{\cong}$) is contained in both $\cat C^h$ and $\cat C^v$;
\item[(iv)] all commutative squares of the form\[
    \begin{tikzcd}
        A \ar[r, >->] \ar[d, ->>, swap, "\cong"] & B \ar[d, ->>, "\cong"]\\
        C \ar[r, >->]& D
        \end{tikzcd} ~\text{ and }~ \begin{tikzcd}
        A \ar[r, >->, "\cong"] \ar[d, ->>]  & B \ar[d, ->>]\\
        C \ar[r, >->, swap, "\cong"]& D
    \end{tikzcd}
    \] are distinguished;
\item[(v)] the object $O$ is initial in both $\cat C^h$ and $\cat C^v$.
\end{enumerate}
\end{definition}

\begin{definition}
    A category with squares is \textit{cocartesian monoidal} if the distinguished object $O$ is the unit for the coproduct and for every pair of objects $A,B\in \cat C$, the squares
\[
    \begin{tikzcd}
        O \ar[r, >->] \ar[d, ->>] \ar[dr, phantom, "\square"] & B \ar[d, ->>]\\
        A \ar[r, >->]& A\amalg B
        \end{tikzcd} ~\text{ and }~ \begin{tikzcd}
        O \ar[r, >->] \ar[d, ->>] \ar[dr, phantom, "\square"] & A \ar[d, ->>]\\
        B \ar[r, >->]& A\amalg B
    \end{tikzcd}
\]
are distinguished.
\end{definition}

All of the categories with squares that we consider in this paper are cocartesian monoidal, and so we will often drop the prefix and just say ``category with squares.'' We impose this condition so that the resulting squares $K$-theory space has some desirable properties (see \cref{rmk:extra conditions}).

\begin{definition}
    Given a category with squares $\cat C$, let $T_n \cat C$ denote the category whose objects are length $n$ sequences of composable horizontal morphisms in $\cat C$
    \[c_{0}\cof c_{1}\cof c_2\cof \dots\cof c_{n}\]
    and whose morphisms are natural transformations of diagrams in $\cat C$.     These categories assemble into a simplicial category $T_\bullet\cat C$, where all faces and degeneracies are like those in a nerve. Let $T^{\square}_\bullet \cat C\subseteq T_\bullet \cat C$ denote the sub-simplicial category whose morphisms are those natural transformations of diagrams which are valued in distinguished squares. That is, a morphism in $T^{\square}_n \cat C$ is a diagram of the form
    \[\begin{tikzcd}
        c_{0}\rar[>->] \dar[->>]\ar[rd, phantom, "\square"]  & c_{1}\rar[>->] \dar[->>]\ar[rd, phantom, "\square"]  & c_{2}\rar[>->] \dar[->>] \ar[rd, phantom, "\square"]& \cdots  \rar[>->]\ar[rd, phantom, "\square"]& c_{n}\dar[->>]\\
        c'_{0}\rar[>->]  & c'_{1}\rar[>->]& c'_{2}\rar[>->]& \cdots  \rar[>->]& c'_{n}
    \end{tikzcd}.\]
\end{definition}

\begin{definition}\label{defn:T-dot}
The \textit{$K$-theory space} of a category with squares $\cat C$ is \[
K^\square(\cat C) = \Omega_O \abs{T^{\square}_\bullet(\cat C)}
,
\] and the \textit{$K$-groups} of $\cat C$ are the homotopy groups of $K^\square(\cat C)$ \[K_i^\square (\cat C) = \pi_i(K^\square(\cat C)).\]
\end{definition}

\begin{remark}\label{rmk:extra conditions}
    The cocartesian monoidal condition on $\cat C$ implies that $K^{\square}(\cat C)$ has the following desirable properties: First, the connected components have the expected description as \[
    K_0^\square(\cat C) \cong \ZZ[{\rm Ob}\cat C]/\sim
    \] where $\sim$ is generated by $[O]=0$ and $[A]+[D] = [B]+[C]$ for every distinguished square\[\begin{tikzcd}
    A \ar[r, >->] \ar[d, ->>] \ar[dr, phantom, "\square"] & B \ar[d, ->>]\\
C \ar[r, >->]& D\end{tikzcd},
    \] see \cite[Theorem 3.1]{CKMZ:squares}; second, $K^{\square}(\cat C)$ is a grouplike $\mathbb{E}_{\infty}$-space \cite[Theorem 2.5]{CKMZ:squares}. Implicitly, here and in the remainder of the paper, we fix the $\mathbb{E}_{\infty}$-operad $\mathcal{O}$ considered in \cite{CKMZ:squares} which $K^{\square}(\cat C)$ is an algebra over.  Whenever we say $\mathbb{E}_{\infty}$-space, we mean an $\mathcal{O}$-algebra.
\end{remark}

\begin{example}\label{exmaple: wald is square}
    The prototypical examples of squares categories are Waldhausen categories. Given a Waldhausen category $\cat C$, we give it a category with squares structure by declaring the distinguished object $O$ to be the zero object of $\cat{C}$, the horizontal morphisms $\cat{C}^h$ to be cofibrations, and the vertical morphisms $\cat{C}^v$ to be all morphisms. The distinguished squares are those squares \[\begin{tikzcd}
    A \ar[r, >->] \ar[d] \ar[dr, phantom, "\square"] & B \ar[d]\\ C \ar[r, >->]& D\end{tikzcd},
    \] so that the natural map $B\cup_A C\to D$ is a weak equivalence in $\cat C$. In this case, $T^{\square}_\bullet \cat C$ is known as the \textit{Thomason construction} and Waldhausen shows in \cite[page 334]{waldhausen} that $K^{\square}(\cat C)\simeq K(\cat C)$. 
\end{example}

A relevant example for our purposes is that of $SK$-manifolds. We review the non-equivariant example below, following \cite{HMMRS}, and recall the equivariant version from \cite{SKG} later in \cref{sec:SKG example}. Further examples of categories with squares may be found, e.g. in \cite{CKMZ:squares} and \cite{callesarazola}.

\begin{example}\label{ex:SK}
The main construction of \cite{HMMRS}, is a cocartesian monoidal category with squares ${\rm Mfld}_d^{\bndry}$. The objects of ${\rm Mfld}_d^{\bndry}$ are compact, orientable $d$-manifolds with boundary.  Both the horizontal and vertical morphisms in ${\rm Mfld}_d^{\bndry}$ are \textit{SK-embeddings}, which are smooth embeddings with an additional condition on the boundary (see \cite[Definition 4.1]{HMMRS} or \cite[Definition 4.4]{SKG} with $G=e$). The cocartesian monoidal structure is disjoint union, the distinguished object is the empty manifold $\varnothing$, and squares are pushout squares. The $K_0^{\square}$-group of  ${\rm Mfld}_d^{\bndry}$ is an $SK$-group for manifolds with boundary, which encodes how manifolds of a fixed dimension may be ``cut up'' and ``pasted'' back together, in the spirit of \cite{KKNO73}. 
\end{example}

The $K$-theory of squares categories is functorial in \textit{square functors}, which are those functors $F\colon \cat C\to \cat D$ which are strong monoidal for the cocartesian moniodal structures and preserve the relevant structure, i.e. the distinguished object and the squares. We let $\catsq$ denote the category of (cocartesian monoidal) categories with squares and square functors. 

\begin{definition}
    A \textit{natural isomorphism of square functors} is a natural isomorphism $\eta\colon F\Rightarrow F'$ which respects the coproduct structure. That is, for $F,F'\colon \cat C\to \cat D$ and any two objects $A,B\in \cat C$, the following diagram commutes:\[\begin{tikzcd}
        F(A) \amalg F(B) \ar[r, "\eta_A \amalg \eta_B"] \ar[d, 
       swap,  "\cong"]& F'(A) \amalg F'(B) \ar[d, "\cong"]\\
        F(A\amalg B) \ar[r, swap, "\eta_{A\amalg B}"] & F'(A\amalg B)
    \end{tikzcd}\]
   where the unlabeled maps come from the assumption that $F$ and $F'$ are strong symmetric monoidal.  Note that this condition is equivalent to saying that $\eta$ is a monoidal natural isomorphism. 
\end{definition}

\begin{remark}
    To define natural transformations of squares functors more generally, one needs to specify either a horizontal or vertical direction (see, e.g., \cite[Definition 2.3]{callesarazola}). Since we only consider natural isomorphisms in this paper, this choice is not needed as we have assumed that the isomorphisms in squares categories are both horizontal and vertical morphisms.
\end{remark}

\begin{remark}
    We can give $\catsq$ the structure of a $(2,1)$-category, where the morphisms in $\catsq(\cat C, \cat D)$ are natural isomorphisms of square functors. Consequently, we may do a change of enrichment along the nerve to obtain a simplicially enriched category of categories with squares, which we denote by $\catsq_{\Delta}$. 

    Similarly, the category of simplicial categories, denoted $\mathrm{sCat}$, is a $(2,1)$-category whose homcategories consist of simplicial functors and simplicial natural isomorphisms. We let $\mathrm{sSym}$ denote the $(2,1)$-category of simplicial symmetric monoidal categories, strong monoidal functors, and monoidal natural isomorphisms. There is an forgetful functor $\mathrm{sSym}\hookrightarrow \mathrm{sCat}$ which forgets the symmetric monoidal structure. 
    Changing enrichment as above results in simplicially enriched forgetful functor $\mathrm{sSym}_{\Delta} \hookrightarrow \mathrm{sCat}_{\Delta}$. 
\end{remark}

Let $\Top_{\Delta}$ denote the simplicially enriched category of topological spaces, with $\Top_{\Delta}(X,Y) = {\rm Sing}({\rm Map}(X,Y))$. Note that since $T_\bullet^{\square}\colon \catsq\to \mathrm{sCat}$ is a strict $2$-functor (as it is essentially a nerve construction), it extends to a simplicially enriched functor $\catsq_{\Delta} \to \mathrm{sCat}_{\Delta}$. The following result is then immediate from the fact that geometric realization $\abs{-}\colon s\Cat_{\Delta}\to \Top_{\Delta}$ is simplicially enriched. 

\begin{proposition}\label{proposition: K square is simplicial}
    The functor $T^{\square}_\bullet \colon \catsq\to s\Cat$ extends to a simplicially enriched functor $\catsq_{\Delta}\to s\Cat_{\Delta}$. Consequently, $K^{\square}\colon \catsq_{\Delta}\to \Top_{\Delta}$ is also a simplicially enriched functor.
\end{proposition}

In fact, the cocartesian monoidal condition on the objects of $\catsq$ and the requirement that squares functors and natural transformations preserve coproducts ensures that the simplicially enriched functor $T_\bullet^{\square}\colon \catsq_\Delta\to \mathrm{sCat}_{\Delta}$ factors through the forgetful functor $\mathrm{sSym}_{\Delta}\rightarrow \mathrm{sCat}_{\Delta}$. A consequence of this observation is the following lemma. 

\begin{lemma}\label{lemma: K square factors through spectra}
    The simplicially enriched functor $K^{\square}\colon \catsq_{\Delta}\to \Top_{\Delta}$ factors through $\mathrm{Alg}^{\mathrm{grp}}_{\mathbb{E}_{\infty}}(\mathrm{Top})_{\Delta}$, the simplicial category of grouplike $\mathbb{E}_{\infty}$-spaces.
\end{lemma}\begin{proof}
It suffices to show that $K^{\square}$ factors through $\Bbb{E}_\infty$-spaces (without the grouplike condition), as each $K^{\square}(\cat C)$ is a grouplike $\Bbb{E}_{\infty}$-space for all $\cat C\in \catsq$ by \cref{rmk:extra conditions}. We may factor $K^{\square}$ as the composite of simplicially enriched functors\[
\catsq_{\Delta} \xrightarrow{T_\bullet^{\square}} (\mathrm{sCat}_*)_{\Delta} \xrightarrow{sB} (\mathrm{sTop}_*)_{\Delta} \xrightarrow{\abs{-}} (\Top_*)_{\Delta} \xrightarrow{\Omega} (\Top_*)_{\Delta},
\] where where the subscript $*$ denotes pointed objects and $sB$ denotes the levelwise application of the classifying space functor. As just observed, $T_\bullet^{\square}$ actually factors through $\mathrm{sSym}_{\Delta}$, and so it suffices to show that the composite $\Omega \circ \abs{-}\circ sB$ sends $\mathrm{sSym}_{\Delta}$ to the simplicial category of $\Bbb{E}_\infty$-spaces. The key result is that the classifying space functor sends monoidal functors to maps of $\Bbb{E}_{\infty}$-spaces and monoidal transformations to homotopies through maps of $\Bbb{E}_{\infty}$-spaces \cite[(2.9)]{Thomason:82}. The claim then follows from the fact that $\abs{-}$ sends simplicial $\Bbb{E}_{\infty}$-spaces to $\Bbb{E}_{\infty}$-spaces and $\Omega$ preserves $\Bbb{E}_{\infty}$-spaces.
\end{proof}

\begin{remark}
    The lemma above says that a length $n$ sequence $F_0 \xRightarrow{\eta_1} \dots \xRightarrow{\eta_n} F_n$ in $\catsq(\cat C, \cat D)$ determines a map of $\Bbb{E}_{\infty}$-spaces $K^{\square}(\cat C)\otimes \Delta^n \to K^{\square}(\cat D)$, where $\otimes$ denotes the tensoring of of an $\mathbb{E}_{\infty}$-space with a space.  In particular, for $n=1$, if $\eta\colon F\Rightarrow F'$ is a natural isomorphism of square functors then $K^{\square}(\eta)$ is a homotopy between $K^{\square}(F)$ and $K^{\square}(F')$ through maps of grouplike $\mathbb{E}_{\infty}$-spaces. 
\end{remark}

\section{Spectral Mackey functors}\label{sec:SpMacks}

In this section we review spectral Mackey functors as a model for genuine $G$-spectra. We make use of the $\infty$-categorical model of Barwick \cite{barwick_spectral_2017, barwick_spectral_2020}, but this could also be done using the model theoretic approach of Guillou--May \cite{GuillouMay}. Throughout this section, we fix a finite group $G$ and write $\Fin^G$ for the category of finite $G$-sets and equivariant functions.

\begin{definition}
    The \emph{Burnside bicategory} $\mathbb{B}^G$ is the bicategory $\Span(\Fin^G)$.  Explicitly, the objects of this bicategory are finite $G$-sets, the objects of $\mathbb{B}^G(X,Y)$ are spans
    \[
        \omega = [X\leftarrow A\rightarrow Y]
    \]
    and a morphism $\omega\to \omega'$ in $\mathbb{B}^G(X,Y)$ is a commutative diagram
    \[
    \begin{tikzcd}[row sep = small]
         & A  \ar[dl] \ar[dr] \ar[dd]&\\
         X & & Y\\
         & A' \ar[ur] \ar[ul] & 
    \end{tikzcd}.
    \]
    We will write $\mathbb{B}^G_{(2,1)}\subseteq \mathbb{B}^G$ for the maximal $(2,1)$-bicategory; that is, $\mathbb{B}^G_{(2,1)}$ has the same objects and $1$-cells as $\mathbb{B}^G$ but only the $2$-cells which are isomorphisms.
\end{definition}
\begin{remark}\label{remark: totally ordered G-sets}
    By carefully picking a model of finite $G$-sets, one can make $\mathbb{B}^G$ into a strict $2$-category.  Namely, if one defines a finite $G$-set to be a a pair $(n,\alpha)$ where $n\in \mathbb{Z}_{\geq 0}$ and $\alpha\colon G\to \Sigma_n$ a group homomorphism, then one can make choices of pullbacks which make the horizontal composition in $\mathbb{B}^G$ strictly associative.  To our knowledge, this process cannot be done in such a way that the horizontal composition is strictly unital, but this can be corrected by adding an additional object to $\mathbb{B}^G$ which is isomorphic to $(0,G\xrightarrow{!}\Sigma_0)$ and serves as a strict unit.  We refer the reader to \cite[Definition 7.2]{BohmannOsorno1} for details regarding the correct choice of pullbacks and \cite[Appendix A]{GuillouMay} for details on the addition of a strict unit.  For the remainder of this paper we will suppose that we have made the correct choices and added a strict unit so that $\mathbb{B}^G$ is a strict $2$-category.
\end{remark}

Recall the Duskin nerve, first constructed in \cite{Duskin}, is a standard functorial construction which produces a simplicial set $N_*^D(\cat{B})$ from a $2$-category $\cat{B}$. In more detail, the Duskin nerve is the simplicial set obtained by first viewing a (locally small) $2$-category $\cat B$ as a simplicial category $\cat B_{\Delta}$ by changing enrichment along the ordinary nerve functor $N_*\colon \Cat\to \sSet$, and then passing from simplicial categories to simplicial sets using the homotopy coherent nerve. The resulting simplicial set $N^{hc}_*(\cat{B}_{\Delta})$ is the Duskin nerve $N_*^D(\cat{B})$.


\begin{remark}
    The Duskin nerve can be defined for any bicategory, not just strict $2$-categories.  In this greater level of generality, the construction described above is inadequate because we cannot change enrichment along the ordinary nerve functor.  Nevertheless, when the input bicategory is actually a strict $2$-category, the two resulting simplicial sets are isomorphic \cite[\href{https://kerodon.net/tag/00KY}{Tag 00KY}]{kerodon}.
\end{remark}

\begin{lemma}
    When $\cat{B}$ is a $(2,1)$-category, the Duskin nerve $N_*^D(\cat{B})$ is an $\infty$-category.
\end{lemma}
\begin{proof}
    Since $\cat{B}$ is a $(2,1)$-category the mapping categories of $\cat{B}$ are groupoids and so the simplicial category $\cat{B}_{\Delta}$ is locally Kan. Thus the homotopy coherent nerve of $\cat{B}_{\Delta}$ is a $\infty$-category.
\end{proof}

To construct spectral Mackey functors, we take the Duskin nerve of the $(2,1)$-category $\mathbb{B}^G_{(2,1)}$, and call the resulting $\infty$-category the \textit{effective Burnside category}.  We note that while this definition is not exactly the one in \cite{barwick_spectral_2017}, it is equivalent (cf. \cite[Definition 2.1]{CMNN}).

\begin{definition}
    The \emph{effective Burnside category} of $G$ is the $\infty$-category $\mathcal{A}^G_{\eff}:=N_*^D(\mathbb{B}^G_{(2,1)})$. 
\end{definition}

Note that the effective Burnside $\infty$-category has categorical products, given by disjoint union of finite $G$-sets. Moreover, the effective Burnside category is a preadditive $\infty$-category, meaning that its homotopy category is preadditive.  To see this, we observe that for any $(2,1)$-category $\cat{B}$, the homotopy category of $N_*^D(\cat{B})$ is equivalent to the homotopy category of $\cat{B}$.  In the case $\cat{B} = \mathbb{B}^G_{(2,1)}$ the resulting homotopy category is the \emph{Lindner category} \cite{Lindner} of finite $G$-sets which is preadditive \cite[Section 3.8]{barwick_spectral_2017}.  

\begin{definition}
    Let $\cat{C}$ be an $\infty$-category with products.  A \emph{Mackey functor in $\cat{C}$} is a product preserving functor $\mathcal{A}^G_{\eff}\to \cat{C}$.  We write $\Mack_G(\cat{C})$ for the $\infty$-category of Mackey functors in $\cat{C}$ and natural transformations between them.  
\end{definition}

When $\cat{C} = \Sp$ is the $\infty$-category of spectra we call a Mackey functor in $\cat{C}$ a \emph{spectral Mackey functor}.  The idea that spectral Mackey functors model genuine $G$-spectra is originally due to Guillou--May, who implement this idea using model categories \cite{GuillouMay}. A proof of the following theorem can be found in \cite{NardinThesis} or in \cite[Appendix A]{CMNN}.

\begin{theorem}
    The homotopy category of $\Mack_G(\Sp)$ is equivalent to the equivariant stable homotopy category.
\end{theorem}

In particular, constructing a spectral Mackey functor is equivalent to producing (the homotopy type of) a genuine $G$-spectrum.

\subsection{Constructing spectral Mackey functors via squares $K$-theory}\label{sec:constructing SpMacks}

In this brief section, we explain a procedure for producing spectral Mackey functors using squares $K$-theory. The key observation is that every product-preserving $\infty$-functor $F\colon \cat{C}\to\cat{D}$ induces a $\infty$-functor $\Mack_G(\cat{C})\to \Mack_G(\cat{D})$ by post-composition.  

\begin{theorem}\label{thm:KT machine}
    Squares $K$-theory induces an $\infty$-functor $K^{\square}_G\colon \Mack_G(N_*^D(\catsq))\to \Mack_G(\Sp)$.
\end{theorem}
\begin{proof}
    First, note that the $(2,1)$-category $\catsq$ has products, given by taking product categories, and thus the left hand side is a well defined $\infty$-category.  It suffices to prove the squares $K$-theory induces a product preserving $\infty$-functor $N_*^D(\catsq)\to \Sp$.  
    

Recall from \cref{lemma: K square factors through spectra} that squares $K$-theory determines a functor $K^{\square}\colon \Cat^{\square}_{\Delta}\to \mathrm{Alg}^{\mathrm{grp}}_{\mathbb{E}_{\infty}}(\mathrm{Top})_{\Delta}$, where the target is the simplicial category of grouplike $\mathbb{E}_{\infty}$-spaces.  Let $L^H(\mathrm{Alg}^{\mathrm{grp}}_{\mathbb{E}_{\infty}}(\mathrm{Top})_{\Delta})$ denote the hammock localization (in the sense of \cite{DK2}) of this simplicial category with respect to the weak equivalences of $\mathbb{E}_{\infty}$-spaces. There is then a natural chain of simplicial functors
\[
    \Cat^{\square}_{\Delta}\xrightarrow{K^{\square}}\mathrm{Alg}^{\mathrm{grp}}_{\mathbb{E}_{\infty}}(\mathrm{Top})_{\Delta}\to L^H(\mathrm{Alg}^{\mathrm{grp}}_{\mathbb{E}_{\infty}}(\mathrm{Top})_{\Delta})\to \mathbf{R}L^H(\mathrm{Alg}^{\mathrm{grp}}_{\mathbb{E}_{\infty}}(\mathrm{Top})_{\Delta})
\]
where $\mathbf{R}$ denotes a fibrant replacement in the Bergner model structure. The (underived) homotopy coherent nerve of the target is the Dwyer--Kan localization of the simplicial model category of grouplike, $\mathbb{E}_{\infty}$-algebras at weak equivalences, which is equivalent to the $\infty$-category of connective spectra (\cite[Proposition 4.8]{DK3}, see also \cite{Hinich}).  Thus, applying the homotopy coherent nerve functor to the composite above yields
\[
    N^D_*(\Cat^{\square})\to \Sp^{\geq0},
\]
and postcomposing with the inclusion $\Sp^{\geq0}\hookrightarrow \Sp$ gives the claimed functor. To see that this functor preserves products, it suffices to check on homotopy categories.  At this level, the claim follows because $K^{\square}$, as a functor to spaces, preserves products up to homeomorphism.
\end{proof}

\begin{corollary}\label{corollary: from mackey functor of squares cats to spectral Mackey}
    Any product-preserving pseudofunctor $\cat{M}\colon \mathbb{B}^G_{(2,1)}\to \catsq$ determines a spectral Mackey functor $K^{\square}_G(\cat{M})$ such that
    \[
        K^{\square}_G(\cat{M})^H\simeq K^{\square}(\cat{M}(G/H))
    \]
    for all $H\leq G$. The restriction, transfer, and conjugation morphisms are equivalent to applying $K^{\square}$ to the restriction, transfer, and conjugation morphisms of $\cat{M}$.
    Additionally, if $\eta\colon \cat M\Rightarrow \cat N$ is a pseudonatural transformation of product-preserving pseudofunctors then $K^{\square}(\eta)\colon K^{\square}_G(\cat M)\to K^{\square}_G(\cat N)$ is a map of genuine $G$-spectra.
\end{corollary}\begin{proof}
    Applying the Duskin nerve to the pseudofunctor $\cat{M}$, we obtain an $\infty$-functor
    \[
        \mathcal{A}^G_{\eff} = N_*^D(\mathbb{B}^G_{(2,1)})\xrightarrow{N_*^D(\cat{M})} N_*^D(\catsq)
    \]
    which is product preserving and hence represents and object in $\Mack_G(\catsq)$.  Post-composing with the functor $K^{\square}_G\colon \Mack_G(N_*^D(\catsq))\to \Mack_G(\Sp)$ from \cref{thm:KT machine} gives the claimed spectral Mackey functor $K^{\square}_G(\cat{M})$.  The fact that this sends pseudonatural transformations to natural transformations of $\infty$-functors follows from the fact that the Duskin nerve sends pseudonatural transformations to simplicial homotopies \cite[Proposition 4]{BFB}.
\end{proof}

In summary, we can produce a spectral Mackey functor from any product-preserving pseudofunctor $\cat{M}\colon \mathbb{B}^G_{(2,1)}\to \catsq$ and we can construct morphisms of spectral Mackey functors from the data of pseudonatural transformations. We note that by restricting along the inclusion $\mathbb{B}^G_{(2,1)}\to \mathbb{B}^G$, we can also produce spectral Mackey functors from product-preserving pseudofunctors $\cat{M}\colon \mathbb{B}^G\to \catsq$. In the next section, we use this approach to produce a genuine $G$-spectrum which encodes the data of equivariant $SK$-manifolds.

\section{The genuine $G$-spectrum of equivariant $SK$-manifolds}\label{sec:SKG example}

We give a variation of the category with squares $\Mfldsq$ considered in \cite{SKG}, which will have an equivalent squares $K$-theory.  We note that this construction is very similar to, and in fact inspired by, the construction in \cite[Section 3]{SKG}.  After constructing these categories with squares, we assemble them into a product preserving pseudofunctor $\mathbb{B}^G\to \catsq$ and apply \cref{corollary: from mackey functor of squares cats to spectral Mackey} to produce a $G$-spectrum for equivariant $SK$-manifolds. We follow the conventions of \cite{SKG}, namely that by $G$-manifold we mean an unoriented compact smooth manifold with boundary
and a smooth $G$-action.

\begin{definition}
    For a finite $G$-set $X$, let $\cat{M}_n^G(X)$ denote the category of pairs $(M,f)$, where $M$ is a compact, smooth $G$-manifold of dimension $n$, possibly with boundary, and $f\colon M\to X$ is an equivariant  continuous map.  A morphism $\phi\colon (M,f)\to (M',f')$ in $\cat{M}^G_n(X)$ is a smooth map $\phi\colon M\to M'$ over $X$.
\end{definition}

A smooth equivariant map $\phi\colon M\to M'$ is an \emph{SK-embedding} if it is an equivariant smooth embedding such that every component of $\partial M$ is either mapped surjectively onto a component of $\partial M'$ or is disjoint from $\partial M'$.  We say a morphism $\varphi\in \cat{M}_n^G(X)$ is an \textit{SK-embedding} if the underlying map of $G$-manifolds is an $SK$-embedding. The proof of the following lemma is essentially the same as the proof of \cite[Proposition 4.3]{HMMRS}.

\begin{lemma}
    For any finite $G$-set $X$, the category $\cat{M}^G_n(X)$ is a category with squares where the horizontal and vertical morphisms are equivariant $SK$-embeddings, the distinguished squares are pushout squares, and the distinguished object is the empty manifold with the unique map to $X$.
\end{lemma}

For any finite group $G$ with subgroup $H\leq G$, let $\Fin^G$ and $\Fin^H$ denote the categories of finite $G$- and $H$-sets, respectively. It is well known that there is an equivalence of categories between $\Fin^H$ and the slice category $(\Fin^G)_{/(G/H)}$.  The next proposition extends idea to $\cat M^G_n(G/H)$.

\begin{proposition}\label{proposition: equivalence of manifold squares cats}
    Let $H\leq G$ be a subgroup.  There is an equivalence of categories with squares
    \[
        \cat{M}_n^G(G/H)\simeq \Mfld^{H,\partial}_n
    \]
    where the right hand side is the category of squares considered in \cite[Definition 4.5]{SKG}.
\end{proposition}
\begin{proof}
    We first define a functor $\Phi\colon \cat{M}_n^G(G/H)\to \Mfld^{H,\partial}_n$. 
    Given a pair $(M,f\colon M\to G/H)\in \cat{M}^G_n(G/H)$, let $\Phi(M) = f^{-1}(eH)$ be the smooth manifold which lives over the unit coset.  Given any $h\in H$ and $x\in \Phi(M)$ we have $f(hx) = hf(x) = h(eH) = eH$ and thus $hx\in \Phi(M)$; in particular the $G$-action on $M$ restricts to an $H$-action on $\Phi(M)$.  On morphisms, $\Phi$ is simply given by restricting morphisms to this preimage. Note that since the restriction of an $SK$-embedding to a component is an $SK$-embedding, the functor $\Phi$ preserves all $SK$-embeddings. With this observation in hand, it is straightforward to see that $\Phi$ is a square functor.
    
    We will show that $\Phi$ is an equivalence by producing an inverse $\Psi\colon \Mfld^{H,\partial}_n\to \cat{M}^G_n(G/H)$, which is just the usual induction of $H$-spaces up to $G$-spaces. Pick coset representatives $g_1,\dots, g_m$ for $G/H$, with $g_1=e$.  For any $g\in G$ and any $1\leq i \leq m$ there exists unique $1\leq j_{g,i}\leq m$ and $h_{g,i}\in H$ such that $gg_i = g_{j_{g,i}}h_{g,i}$.
    Now, for $1\leq i\leq m$ and $M\in \cat M^G_n(G/H)$, let 
    \[
        g_iM = \{(g_i,m)\mid m\in M\}
    \]
    where $g_i$ is acting as a formal label. The set $g_iM$ can be given the structure of a smooth manifold diffeomorphic to $M$. Let 
    \[
        \Psi(M) = \coprod_{i=1}^n g_iM
    \]
    with $G$-action given by
    \[
        g\cdot (g_i,m) = (g_{j_{g,i}},h_{g,i}m)\in g_{j_{g,i}}M.
    \]
    If $\pi\colon \Psi(M)\to G/H$ is the map which sends all of $g_iM$ to $g_iH$ then the pair $(\Psi(M),\pi)$ is a well defined object in $\cat{M}^G_n(G/H)$. It is straightforward to verify $\Psi$ is a square functor.
    
    Now observe that, since $g_1 = e$, we have an $H$-diffeomorphism $g_1M \cong M$ and thus $\Phi\Psi\cong \mathrm{id}$. Next, to see $\Psi\Phi\cong \mathrm{id}$, take an arbitrary pair $(M,f)\in \cat{M}^G_n(G/H)$ and define a map $\Psi\Phi(M,f)\to (M,f)$ by
    \[
        (g_i,m)\mapsto g_im.
    \]
    Locally, this map factors as 
    \[
        \pi^{-1}(g_iH)\cong \Phi(M,f)\cong f^{-1}(eH)\hookrightarrow M\xrightarrow{g_i} M,
    \]
    and we see that its derivative at any point gives an isomorphism of tangent spaces and hence is a local diffeomorphism.  This map is also a bijection and therefore is a diffeomorphism.
\end{proof}

This proposition, together with \cite[Theorem 4.8]{SKG}, yields the following corollary. See \cite[Definition 2.4]{SKG} for a precise definition of $\SK^H_n$. 

\begin{corollary}\label{cor:get SKH}
    For all $H\leq G$, there is an isomorphism of abelian groups
    \[
        K_0^{\square}(\cat{M}^G_n(G/H))\cong \SK^H_n,
    \]
    where the right hand side is the $H$-equivariant $SK$-group of $n$-manifolds with boundary.
\end{corollary}

In the remainder of this section, we assemble the necessary results to prove the following theorem. 

\begin{theorem}\label{thm:main}
    The assignment $X\mapsto \cat M^G_n(X)$ extends to a product preserving pseudofunctor $\cat M^G_n\colon \mathbb{B}^G\to \catsq$. Therefore there is a spectral Mackey functor $K^{\square}(\cat M^G_n)\in \Mack_G(\Sp)$ with
    \[
        \pi_0^H(K^{\square}(\cat M^G_n))\cong K_0^{\square}(\cat{M}^G_n(G/H))\cong \SK^H_n
    \]
    for all $H\leq G$.
\end{theorem}

The second claim follows from \cref{thm:KT machine} and \cref{cor:get SKH}, and so the the work of this section is proving the first claim. As a consequence of this theorem, we obtain a positive answer to the first part of the conjecture posed in the introduction of \cite{SKG}.

\begin{corollary}
    The connected components $\pi_0 K^{\square}(\cat M^G_n)$ recover the $\SK^G$-Mackey functor from \cite{SKG}.
\end{corollary}\begin{proof}
    This follows from \cref{thm:main,cor:get SKH} and comparing the construction of $\cat M^G_n$ to the one in \cite[Section 3]{SKG}. In particular, the definition of the restriction and transfer functors are identical.
\end{proof}

We now construct the pseudofunctor $\cat M_n^G$. Given any map of finite $G$-sets $r\colon X\to Y$, there is an induced functor
\[
    r^*\colon \cat{M}^G_n(Y)\to \cat{M}^G_n(X)
\]
given by pullback. Explicitly, given a pair $(M,f\colon M\to X)$, we have a pullback square
\[
    \begin{tikzcd}
    P \ar[d,"r_f"'] \ar[r,"f_r"] & X \ar[d,"r"]\\
    M \ar[r,"f"] & Y
    \end{tikzcd}
\]
and we define $r^*(M,f) = (P,f_r)$.  We note that while the class of smooth $n$-manifolds with boundary is not, in general, closed under the operation of pullback, there is no issue here because we are pulling back along a map of finite $G$-sets. The functor $r^*$ is determined on morphisms by the universal property of pullbacks. We now show that $r^*$ preserves $SK$-embeddings.

\begin{lemma}\label{lemma: restriction preserves SK embeddings}
    Let $\varphi\colon (M,f)\to (M',f')$ be an $SK$-embedding in $\cat{M}^G_n(Y)$ and $r\colon X\to Y$ an equivariant map of finite $G$-sets. Then $r^*(\varphi)\colon (P,f_r)\to (P',f'_r)$ is an $SK$-embedding in $\cat{M}^G_n(X)$.
\end{lemma}
\begin{proof}
    Since being an $SK$-embedding is a condition on the underlying non-equivariant map, it suffices to prove the claim for $G=e$. Factoring $r\colon X\to Y$ as $X\twoheadrightarrow \im(r)\hookrightarrow Y$, we see that it suffices to prove the claim in the special case that $r$ is (i) injective and (ii) surjective.

    Let $M_y = f^{-1}(y)$ and $M_y' = (f')^{-1}(y)$, for $y\in Y$. (i) If $r\colon X\to Y$ is an injection, then 
    \[
        P = r^*(M) = \coprod_{x\in X} M_{r(x)} \subset M
    \]
    and $f_r$ is the restriction of $f$ to these components.  The pullback $P' = r^*(M')$ has a similar description, and the induced map $r^*(\phi)\colon P\to P'$ is the restriction of $\varphi$ to $P$.  Since the restriction of an $SK$-embedding is an $SK$-embedding, we see $r^*$ preserves $SK$-embeddings when $r$ is injective.

    (ii) When $r$ is surjective, we have 
    \[
        P \cong \coprod_{y\in Y} \coprod_{x\in r^{-1}(y)} M_{x},\quad \mathrm{and}\quad P' \cong \coprod_{y\in Y} \coprod_{x\in r^{-1}(y)} M_{x}'
    \]
    and with this identification the map $r^*(\phi)$ is given by
    \[
        \coprod_{y\in Y}\coprod_{x\in r^{-1}(y)} \varphi \colon \coprod_{y\in Y}\coprod_{x\in r^{-1}(y)} M_{x}\to \coprod_{y\in Y}\coprod_{x\in r^{-1}(y)} M_{x}'
    \]
    which makes sense because the requirement that $f'\circ\varphi =f$ implies that $\varphi(M_x)\subset M'_x$.  Since the disjoint union of $SK$-embeddings is an $SK$-embedding, we are done.
\end{proof}

We claim that $r^*$ also preserves squares. We omit the proof as it is essentially the same as the previous one. The crucial claim is that pullbacks along maps of finite $G$-sets preserve pushout squares, which follows from the observation that $r^*$ is both a left and a right adjoint.

\begin{lemma}\label{lemma: pullback preserves squares}
    If 
    \[
        \begin{tikzcd}
            M \ar[r] \ar[d] & N \ar[d]\\
            P \ar[r] & Q \ar[phantom,from =1-1,to = 2-2,"\square"]
        \end{tikzcd}
    \]
    is a distinguished square (a pushout) in $\cat{M}_n^G(Y)$ and $r\colon X\to Y$ is a map of finite $G$-sets, then 
    \[
        \begin{tikzcd}
            r^*M \ar[r] \ar[d] & r^*N \ar[d]\\
            r^*P \ar[r] & r^*Q \ar[phantom,from =1-1,to = 2-2,"\square"]
        \end{tikzcd}
    \]
    is a distinguished square in $\cat{M}^G_n(X)$.
\end{lemma}

We now describe the transfers. For any map of finite $G$-sets $r\colon X\to Y$, the map $r^*\colon \cat{M}^G_n(Y)\to \cat{M}_n^G(X)$ has a left adjoint $r_!\colon \cat{M}^G_n(X)\to \cat{M}^G_n(Y)$ which, on objects, sends a pair $(M,f\colon M\to X)$ to $(M,r\circ f)$.  On morphisms, this functor sends $\varphi\colon (M,f)\to (M',f')$ to $\varphi\colon (M,r\circ f)\to (M',r\circ f')$ and this evidently preserves $SK$-embeddings.  Since it is a left adjoint, $r_!$ preserves pushouts and hence these left adjoints are square functors.

\begin{remark}
In what follows, it will be helpful to have an explicit presentation of the unit and counit for the adjunction $r_!\dashv r^*$ associated to a map of finite $G$-sets $r\colon X\to Y$.  To that end, fix pairs $(M,f\colon M\to X)\in \cat{M}_n^G(X)$ and $(N,h\colon N\to Y)\in \cat{M}_n^G(Y)$.  
To describe the counit, observe that $r^*(N,h) =(P_N,h_r\colon P_N\to X)$ is defined by the pullback
\[
    \begin{tikzcd}
        P_N \ar[r,"h_r"] \ar[d,"r_h"] & X\ar[d,"r"]\\
        N\ar[r,"h"] & Y
    \end{tikzcd}
\]
and the counit is simply the map $r_h\colon (P_N,r\circ h_r)\to (N,h)$.  
For the unit, consider the diagram
\[\begin{tikzcd}
	M \\
	& {P_M} && X \\
	& M && Y
	\arrow["\eta", dashed, from=1-1, to=2-2]
	\arrow["f", curve={height=-12pt}, from=1-1, to=2-4]
	\arrow["{=}"', curve={height=12pt}, from=1-1, to=3-2]
	\arrow["{(r\circ f)_r}"', from=2-2, to=2-4]
	\arrow[from=2-2, to=3-2]
	\arrow["\lrcorner"{anchor=center, pos=0.125}, draw=none, from=2-2, to=3-4]
	\arrow["r", from=2-4, to=3-4]
	\arrow["{r\circ f}"', from=3-2, to=3-4]
\end{tikzcd}\]
where the pullback square defines $r^*r_!(M,f) = (P_M,(r\circ f))$.  The map $\eta$ is the unit.
\end{remark}

The next proposition is the key technical point in the proof of \cref{thm:main}.

\begin{proposition}\label{proposition: Beck condition}
    For any pullback square
    \[
        \begin{tikzcd}
            A \ar[r,"f"] \ar[d,swap,"g"] & B \ar[d,"h"]\\
            C\ar[r, swap, "k"] & D
        \end{tikzcd}
    \]
    in finite $G$-sets the square
    \[
        \begin{tikzcd}
            \cat{M}^G_n(A) \ar[d,"g_!"] & \cat{M}^G_n(B) \ar[l,"f^*"']  \ar[d,"h_!"]\\
            \cat{M}^G_n(C) & \cat{M}^G_n(D) \ar[l,"k^*"] \ar[from = 1-1,to = 2-2,Rightarrow,shorten = 2mm,"\beta"]
        \end{tikzcd}
    \]
    commutes up to the natural Beck--Chevalley transformation $\beta$ given by
    \[
        g_!f^* \xrightarrow{g_!f^* \cdot \eta_h} g_!f^*\circ(h^*h_!)  = (g_!g^*)\circ k^*h_! \xrightarrow{\epsilon_g \cdot k^*h_!} k^*h_!.
    \]
    Moreover the Beck--Chevalley transformation is an natural isomorphism.
\end{proposition}
\begin{remark}
    The Beck--Chevalley natural transformation $\beta$ in the second square always exists, so the content of the proposition is that this natural transformation is a natural isomorphism.
\end{remark}
\begin{proof}
    We will show that the component of this natural transformation $\beta$ at an arbitrary pair $(M,\alpha)\in \cat{M}_n^G(B)$ is an isomorphism.  Since a morphism in $\cat{M}_n^G(C)$ is an isomorphism if and only if the underlying map in the category of smooth manifolds is an isomorphism, we will simply check this map is a diffeomorphism of smooth manifolds.  

    On underlying manifolds, we have 
    \begin{itemize}
        \item $g_!f^*(M) = M\times_B A$,
        \item $g_!f^*h^*h_!(M) = (M\times_D B)\times_B A$,
        \item $g_!g^*k^*h_!(M) = (M\times_D C)\times_C A$,
        \item $k^*h_!(M) = M\times_D C$
    \end{itemize}
    where, when relevant, we consider $M$ as a manifold over $D$ via the composite $M\xrightarrow{\alpha} B\xrightarrow{h} D$. 
    
    Let $(m,a)\in M\times_B A$ so that $\alpha(m) = f(a)$. Note that since the square $ABCD$ is a pullback, the pasting lemma for pullbacks implies that $M\times_B A\cong M\times_D C$.  This diffeomorphism is given by $(m,a)\mapsto (m,g(a))$. Unpacking the unit and counit of the adjunctions as described above, we can compute the Beck--Chevalley map on the pair $(m,a)$ by the sequence
    \[
        (m,a)\mapsto ((m,\alpha(m)),a) = ((m,g(a)),a)\mapsto (m,g(a))
    \]
    which we have just observed is an diffeomorphism.
\end{proof}

Given a composable pair of equivariant functions
\[
    X\xrightarrow{r} Y \xrightarrow{s} Z,
\]
it is not true that $r^*\circ s^* = (s\circ r)^*$ with strict equality, and hence $X\mapsto \cat{M}_n^G(X)$ is not a functor $(\Fin^G)^{\op}\to \catsq$.  However, there are canonical isomorphisms $r^*\circ s^* \cong (s\circ r)^*$ and these are sufficiently compatible to see that $X\mapsto \cat{M}_n^G(X)$ is a pseudofunctor $\cat{M}^G_n\colon (\Fin^G)^{\op}\to \catsq.$

\begin{lemma}\label{lemma: product preserving}
    The assignment $X\mapsto \mathcal{M}_n^G(X)$ extends to a pseudofunctor $(\Fin^G)^{\op}\to \catsq$ which sends disjoint unions of finite $G$-sets to products of categories.
\end{lemma}
\begin{proof}
    All we need to check is that this psuedofunctor preserves products.  The categorical product in $(\Fin^{G})^{\op}$ is the coproduct in $\Fin^G$, hence it is disjoint union.  If $X = X_1\amalg X_2$ there is an isomorphism 
    \[
        \cat{M}_n^G(X)\cong \cat{M}_n^G(X_1)\times \cat{M}_n^G(X_2)
    \]
    which sends a pair $(M,f\colon M\to X)$ to the tuple $((f^{-1}(X_1),f),(f^{-1}(X_2),f)$.
\end{proof}

Let $\iota\colon (\Fin^G)^{\op}\to \mathbb{B}^G$. denote the embedding $\iota(r\colon X\to Y) = [Y \xleftarrow{r} X\xrightarrow{=} X]$.

\begin{theorem}\label{theorem: there is a pseudofunctor}
    There exists an essentially unique product preserving pseudofunctor
    \[
        \widehat{\cat{M}}_n^G\colon \mathbb{B}^G\to \catsq
    \]
    with the property that $\iota^*(\widehat{\cat{M}}_n^G) = \cat{M}^G_n$. 
\end{theorem}
\begin{proof}
    By (the dual of) \cite[Theorem 2.8(3)]{DPP}, there exists an essentially unique pseudofunctor $\widehat{\cat{M}}_n^G$ with $\iota^*(\widehat{\cat{M}}_n^G) = \cat{M}^G_n$ if and only if $\cat{M}^G_n$ has the property that it sends every pullback of finite $G$-sets to a square in $\catsq$ with invertible Beck--Chevalley transformation.  We confirmed that $\cat{M}^G_n$ has this property in \cref{proposition: Beck condition}.  The fact that $\widehat{\cat{M}}^G_n$ preserves products follows at once from the second half of \cref{lemma: product preserving}.
\end{proof}

\cref{thm:main} is an immediate consequence of \cref{theorem: there is a pseudofunctor,corollary: from mackey functor of squares cats to spectral Mackey}.

\section{Lifting the equivariant Euler characteristic}\label{sec:Euler char}

The authors of \cite{SKG} show that the $H$-equivariant Euler characteristic lifts to maps of spectra $K^{\square}_G(\cat M^G_n)^H\to  K_G(\underline{\ZZ})^H$ for all $H\leq G$, where $K_G(\underline{\ZZ})$ is the equivariant algebraic $K$-theory of the constant coefficient system of $\ZZ$, as defined in \cite{linearization}.  Moreover, on $\pi_0$, the authors show these maps assemble into a map of $G$-Mackey functors.  The main theorem of this section says that these maps of spectra assemble coherently into a map of spectral Mackey functors.
\begin{theorem}\label{theorem: lift euler char}
    There is a map of genuine $G$-spectra $K^{\square}_G(\cat{M}^G_n)\to K_G(\underline{\ZZ})$ which recovers the map of \cite[Theorem B]{SKG} upon taking $H$-fixed points.
\end{theorem}
 
Following \cite[Section 4.3]{SKG}, we will prove this result by constructing a map of genuine $G$-spectra $K^{\square}_G(\cat M^G_n)\to A_G(*)$, where $A_G(*)$ is the genuine equivariant $A$-theory of a point \cite{MalkiewichMerling:Atheory}, and then post-composing with the equivariant linearization map $L\colon A_G(*)\to K_G(\underline{\ZZ})$ from \cite{linearization}. In more detail, the maps $K^{\square}_G(\cat M^G_n)^H\to  K_G(\underline{\ZZ})^H$ constructed in \cite{SKG} factor as 
\[
    K^{\square}_G(\cat{M}^G_n)^H\xrightarrow{\sim} K^{\square}(\Mfld^{G,\partial}_n)\xrightarrow{} A_G(*)^H\xrightarrow{L^H} K_G(\underline{\ZZ})^H
\]
and our approach to \cref{theorem: lift euler char} is to show that the maps $\alpha_H\colon K^{\square}_G(\cat{M}_n^G)^H\to A_G(*)^H$ assemble into a map of genuine $G$-spectra.  

We begin by recalling the definition of the spectrum $A_G(*)$ and the maps $\alpha_H$.  For $H\leq G$, let $R^H$ denote the category of finitely dominated pointed $H$-spaces and equivariant continuous maps.  The category $R^H$ is a Waldhausen category with cofibrations and weak equivalences obtained from the usual genuine model structure on pointed $H$-spaces, and the spectrum $A_G(*)^H$ is defined to be the Waldhausen $K$-theory of $R^H$.

To more easily compare $A_G(*)$ with the $K^{\square}_G$-theory of $\cat M^G_n$, we introduce a variant of $R^H$ as follows. For any finite $G$-set $X$, let $\cat{R}^G(X)$ denote the category of retractive $G$-spaces over $X$.  This is also a Waldhausen category, with weak equivalence and cofibrations inherited from the model structure on retractive $G$-spaces (see \cite[Definition 4.13]{MalkiewichMerling:Atheory}). Mimicking the proof of \cref{proposition: equivalence of manifold squares cats} almost exactly, one proves the following result.

\begin{proposition}\label{proposition: parametrized retractive spaces}
    For any $H\leq G$, there is an exact equivalence of Waldhausen categories $\cat{R}^G(G/H)\simeq R^H$.
\end{proposition}

Given a map of finite $G$-sets $r\colon X\to Y$, we then obtain an exact functor $r^*\colon \cat{R}^G(Y)\to \cat{R}^G(X)$, defined by pullback
\[
    \begin{tikzcd}
        r^*(M) \ar[r] \ar[d] & M \ar[d]\\
        X\ar[r,"r"] & Y 
    \end{tikzcd}
\]
for any retractive space $M\in \cat{R}^G(Y)$.  The functor $r^*$ always has a left adjoint $r_!\colon \cat{R}^G(X)\to \cat{R}^G(Y)$ defined via pushout
\[
\begin{tikzcd}
    X \ar[r,"r"] \ar[d] & Y \ar[d]\\
    N \ar[r] & r_!(N)
    \end{tikzcd}
\]
for any $N\in \cat{R}^G(X)$ (see \cite[Section 2.1]{MaySigurdsson} for that proof that these are in fact adjoint functors).  The following lemma follows from analysis similar to that in the proof of \cref{lemma: restriction preserves SK embeddings}.  We omit the proofs for brevity.

\begin{lemma}
    The functors $r^*$ and $r_!$ are exact functors of Waldhausen categories.
\end{lemma}

Just as for $\cat{M}^G_n$, it is straightforward to see that the assignment $X\mapsto \cat{R}^G(X)$ is a pseudofunctor $(\Fin^G)^{\op}\to \Wald$.  Extending along the inclusion of Waldhausen categories into categories with squares, via \cref{exmaple: wald is square}, we obtain a product preserving pseudofunctor $\cat{R}^G \colon (\Fin^G)^{\op}\to \catsq$.
\begin{proposition}\label{proposition: there is another pseudofunctor}
    The functor $\cat{R}^G\colon (\Fin^G)^{\op}\to \catsq$ extends, essentially uniquely, to a product preserving pseudofunctor $\widehat{\cat{R}}^G\colon \mathbb{B}^G\to \catsq$.
\end{proposition}
\begin{proof}
    As in the proof of \cref{theorem: there is a pseudofunctor}, this follows from \cite[Theorem 2.8(3)]{DPP} as soon as we know the the Beck--Chevalley transformation is a natural isomorphism.  One can either prove this directly, following the proof of \cref{proposition: Beck condition}, or simply refer to \cite[Proposition 2.2.11]{MaySigurdsson}.
\end{proof}

Combining \cref{corollary: from mackey functor of squares cats to spectral Mackey,proposition: there is another pseudofunctor} yields a spectral Mackey functor we denote $A^{\square}_G(*)$.  The equivalences of \cref{proposition: parametrized retractive spaces} show that the fixed points of this spectrum agree with those of the fixed points of equivariant $A$-theory of Malkiewich--Merling. Moreover, through these equivalences the restriction, transfer, and conjugation maps on either spectral Mackey functor are identified.

\begin{proposition}
    The spectrum $A^{\square}_G(*)$ is isomorphic in the equivariant stable homotopy category to $A_G(*)$.
\end{proposition}\begin{proof}[Proof sketch]
     By \cref{exmaple: wald is square}, we could equivalently model $A^{\square}_G(*)$ using Waldhausen $K$-theory instead of squares $K$-theory. In particular, we can observe that $\cat R^G$ factors through Waldhausen categories, and therefore we may post-compose with Waldhausen's $K$-theory functor to obtain a spectral Mackey functor which is equivalent to $A^{\square}_G(*)$. The remaining technical piece of the argument is then to compare this construction, which is done using $\infty$-categories, with the construction of Malkiewich--Merling which uses multicategories. The fact that the $\infty$-categorical and multicategorical constructions agree for equivariant Waldhausen $K$-theory will appear in forthcoming work of the authors and should not be surprising, since the two approaches are essentially the same (this is shown, for instance, for Segal $K$-theory in \cite{CCP}).
\end{proof}

Finally, observe that a pair $(M,f\colon M\to X)\in \cat{M}^G_n(X)$ determines a retractive $G$-space 
\[
    M\amalg X\xrightarrow{f\amalg \mathrm{id}_X} X
\]
where the section is given by the inclusion $X\to M\amalg X$. This assignment extends to a functor $\alpha_X\colon \cat{M}^G_n(X)\to \cat{R}^G(X)$ which preserves pushouts. Because $SK$-embeddings are cofibrations, this is a functor of categories with squares.  Moreover, by the very definition of the restrictions, we have $\alpha_X \circ r^* = r^*\circ \alpha_Y$ for any equivariant map $r\colon X\to Y$. Thus we have a pseudonatural transformation of pseudofunctors $\cat{M}^G_n\Rightarrow \cat{R}^G$. 

\begin{proposition}
    The pseudonatural transformation $\alpha$ extends to a pseudonatural transformation $\widehat{\alpha}\colon \widehat{\cat{M}}^G_n\Rightarrow \widehat{\cat{R}}^G$.
\end{proposition}
\begin{proof}
    By \cite[Theorem Theorem 2.8(2)]{DPP} it suffices to check that for any map of finite $G$-sets $r\colon X\to Y$ that the Beck--Chevalley transformation
\[\begin{tikzcd}
	{\cat{M}^G_n(X)} & {\cat{M}^G_n(X)} & {\cat{R}^G(X)} & {\cat{R}^G(X)} \\
	{\cat{M}^G_n(Y)} & {\cat{M}^G_n(Y)} & {\cat{R}^G(Y)} & {\cat{R}^G(Y)}
	\arrow[""{name=0, anchor=center, inner sep=0}, equals, from=1-1, to=1-2 ]
	\arrow["{r_!}"', from=1-1, to=2-1]
	\arrow["{\alpha_X}", from=1-2, to=1-3]
	\arrow[""{name=1, anchor=center, inner sep=0}, equals, from=1-3, to=1-4]
	\arrow["{r_!}", from=1-4, to=2-4]
	\arrow[""{name=2, anchor=center, inner sep=0}, equals, from=2-1, to=2-2]
	\arrow["{r^*}", from=2-2, to=1-2]
	\arrow["{\alpha_Y}", from=2-2, to=2-3]
	\arrow["{r^*}", from=2-3, to=1-3]
	\arrow[""{name=3, anchor=center, inner sep=0}, equals, from=2-3, to=2-4]
	\arrow["\eta"', Rightarrow, from=0, to=2, shorten  =4mm]
	\arrow["\epsilon",  Rightarrow, from=1, to=3,shorten  =4mm]
\end{tikzcd}\]
    is a natural isomorphism.  But since the functors $r_!$ and $r^*$ on $\cat{M}^G_n$ and $\cat{R}^G$ coincide after the addition of disjoint basepoints via $\alpha$ this natural transformation is the identity, by the triangle identities for the adjunction $r_!\dashv r^*$ between $\cat{R}^G(X)$ and $\cat{R}^G(Y)$.
\end{proof}
Applying \cref{corollary: from mackey functor of squares cats to spectral Mackey}, we obtain a map of genuine $G$-spectra
\[
    \alpha\colon K^{\square}_G(\cat{M}^G_n)\to K^{\square}_G(\cat{R}^G) \simeq A_G(*)
\]
which, upon taking $H$-fixed points, recovers the maps $\alpha_H$.  Thus, extending along the linearization map $A_G(*)\to K_G(\underline{\mathbb{Z}})$ from \cite[Theorem C]{linearization}, this proves \cref{theorem: lift euler char}.

\printbibliography
\end{document}

%% file: shortcuts.tex

\newcommand{\RR}{\mathbb{R}}
\newcommand{\QQ}{\mathbb{Q}}
\newcommand{\NN}{\mathbb{N}}
\newcommand{\ZZ}{\mathbb{Z}}
\newcommand{\CC}{\mathbb{C}}
\newcommand{\PP}{\mathbb{P}}
\newcommand{\HH}{\mathbb{H}}
\newcommand{\EE}{\mathbb{E}}

\newcommand{\abs}[1]{\left\lvert#1\right\rvert}
\newcommand{\ip}[2]{\left\langle #1 , #2\right\rangle}
\newcommand{\norm}[1]{\left\lvert\left\lvert#1\right\rvert\right\rvert}
\newcommand{\gen}[1]{\left\langle #1 \right\rangle}
\newcommand{\set}[1]{\left\{ #1 \right\}}
\newcommand{\cat}[1]{\mathscr{#1}}
\newcommand{\bcat}[1]{{\bf #1}}

\newcommand{\normal}{\trianglelefteq}
\newcommand{\grad}{\nabla}
\newcommand{\Loop}{\Omega}
\newcommand{\codim}{{\rm codim}}
\newcommand{\bndry}{\partial}
\newcommand{\dd}{\partial}
\newcommand{\im}{\operatorname{im}}
\newcommand{\op}{\operatorname{op}}
\newcommand{\colim}{\operatorname{colim}}
\newcommand{\rk}{\operatorname{rk}}
\newcommand{\ob}{\operatorname{Ob}}
\renewcommand{\hom}{\operatorname{Hom}}
\newcommand{\cod}{\operatorname{cod}}
\newcommand{\dom}{\operatorname{dom}}
\newcommand{\id}{\operatorname{id}}
\newcommand{\1}{\mathbf{1}}
\newcommand{\GL}{\operatorname{GL}}
\newcommand{\Alg}{\operatorname{Alg}}
\newcommand{\Map}{\operatorname{Map}}
\newcommand{\Sing}{\operatorname{Sing}}
\newcommand{\quot}{\twoheadrightarrow}
\newcommand{\cof}{\rightarrowtail}
\newcommand{\Mod}{\mathcal{M}od}
\newcommand{\Set}{\mathrm{Set}}
\newcommand{\FP}{\mathrm{FP}}
\newcommand{\Perf}{\mathrm{Perf}}
\newcommand{\fd}{\mathrm{fd}}
\newcommand{\Ab}{\mathrm{Ab}}
\newcommand{\Loc}{\mathrm{Loc}}
\newcommand{\Cat}{\mathrm{Cat}}
\newcommand{\Gpd}{\mathrm{Gpd}}
\newcommand{\Fun}{\mathrm{Fun}}
\newcommand{\Span}{\mathrm{Span}}
\newcommand{\Top}{\mathrm{Top}}
\newcommand{\Sp}{\mathrm{Sp}}
\newcommand{\Wald}{\mathrm{Wald}}
\newcommand{\wex}{\mathrm{wex}}
\newcommand{\ex}{\mathrm{ex}}
\newcommand{\cn}{\mathrm{cn}}
\newcommand{\eff}{\mathrm{eff}}
\newcommand{\pr}{\mathrm{Pr}}
\newcommand{\sSet}{\mathrm{sSet}}
\newcommand{\Coeff}{\mathrm{Coeff}}
\newcommand{\tr}{\mathrm{tr}}
\newcommand{\res}{\mathrm{res}}
\newcommand{\Perm}{\mathrm{Perm}}
\newcommand{\Seg}{\mathrm{Seg}}
\newcommand{\Mack}{\mathrm{Mack}}
\newcommand{\ps}{\mathrm{ps}}
\newcommand{\Proj}{\mathrm{Proj}}
\newcommand{\Hom}{\mathrm{Hom}}
\newcommand{\Mfld}{\mathrm{Mfld}}
\newcommand{\Mfldsq}{\mathrm{Mfld}^{G,\partial}_n}
\newcommand{\SK}{\mathrm{SK}}
\newcommand{\catsq}{\Cat^{\square}}
\newcommand{\Fin}{\mathrm{Fin}}
\newcommand{\spaces}{\cat{S}}
\newcommand{\calg}{\mathrm{Alg}_{\mathbb{E}_{\infty}}}